\documentclass[oneside,11pt,reqno]{amsart}
\usepackage{amssymb,amsmath,amsthm,bbm,enumerate,mdwlist,url,multirow,hyperref,amsthm,stmaryrd,bigints}
\usepackage[shortlabels]{enumitem}
\usepackage[pdftex]{graphicx}
\theoremstyle{definition}
\newtheorem{definition}{Definition}
\numberwithin{definition}{section}
\theoremstyle{theorem}
\newtheorem{proposition}[definition]{Proposition}
\newtheorem{lemma}[definition]{Lemma}
\newtheorem{theorem}[definition]{Theorem}
\newtheorem{corollary}[definition]{Corollary}

\numberwithin{equation}{section}
\theoremstyle{remark}
\newtheorem{remark}[definition]{Remark}
\newtheorem{question}[definition]{Question}

\def\PP{\mathsf{P}}
\def\expp{\mathsf{e}}
\def\EE{\mathsf{E}}
\def\QQ{\mathsf{Q}}
\def\FF{\mathcal{F}}
\def\nn{\mathsf{n}}
\def\GG{\mathcal{G}}
\def\Wq{W^{(q)}}

\def\AA{\mathcal{A}}
\def\BB{\mathcal{B}}
\def\dd{\mathrm{d}}
\def\Wp{W^{(p)}}
\def\Zqt{Z^{(q,\theta)}}
\def\DD{\mathsf{G}}
 \def\WW{\mathcal{W}^{(q)}_{\alpha,p}}
 \def\ZZ{\mathcal{Z}^{(q,\theta)}_{\alpha,p}}
 \def\ZZq{\mathcal{Z}^{(q)}_{\alpha,p}}
\def\drift{\delta}

\def\diffusion{\sigma^2}
\def\II{\mathcal{I}}
 
\makeatletter
\@namedef{subjclassname@2020}{%
  \textup{2020} Mathematics Subject Classification}
\makeatother

\begin{document}
\title[Exit problems for pssMp with one-sided jumps]{Exit problems for positive self-similar Markov processes with one-sided jumps}

\author{Matija Vidmar}
\address{Department of Mathematics, Faculty of Mathematics and Physics, University of Ljubljana, Slovenia}
\email{matija.vidmar@fmf.uni-lj.si}

\begin{abstract}
A systematic exposition of scale functions is given for positive self-similar Markov processes (pssMp) with one-sided jumps. The scale functions express as  convolution series of the usual scale functions associated with spectrally one-sided L\'evy processes that underly the pssMp through the Lamperti transform. This theory is then brought to bear on solving the spatio-temporal: (i) two-sided exit problem; (ii)  joint first passage problem upwards for the  the pssMp and its multiplicative drawdown (resp. drawup) in the spectrally negative (resp. positive) case. 
\end{abstract}

\thanks{Financial support from the Slovenian Research Agency is acknowledged (programmes Nos. P1-0222 \& P1-0402).}



\maketitle

\section{Introduction}
Recent years have seen renewed interest in the (fluctuation) theory of positive self-similar Markov processes (pssMp), that is to say, modulo technicalities, of $(0,\infty)$-valued strong Markov processes $Y=(Y_s)_{s\in [0,\infty)}$ with $0$ as a cemetery state that enjoy the following scaling property: for some $\alpha\in (0,\infty)$, and then all $\{c,y\}\subset (0,\infty)$, the law of $(cY_{sc^{-\alpha}})_{s\in [0,\infty)}$ when issued from $y$ is that of $Y$ when issued from $cy$. See the papers \cite{ckr,pierre,caballero,pardo,chaumont,patie,baurdoux} among others.

On the other hand it is by now widely recognized that the fundamental, and as a consequence a great variety of other non-elementary, exit problems  of upwards or downwards skip-free  strong Markov real-valued processes can often be parsimoniously expressed in terms of a collection of so-called ``scale functions'', be it in discrete or continuous time or space. See e.g. the papers \cite{kkr} for the case of spectrally negative L\'evy processes (snLp), \cite{avram-vidmar} for upwards skip-free random walks, \cite{vidmar:fluctuation} for their continuous-time analogues, \cite{ivanovs,Kyprianou2008} for Markov additive processes and \cite[Proposition~VII.3.2]{revuz-yor}  \cite{zhang,lehoczky} for diffusions. (Note that for processes with stationary independent increments the results for the upwards-skip-free case yield at once also the analogous results for the downwards-skip-free case, ``by duality'' (taking the negative of the process).) This is of course but a flavor of the huge body of literature on a variety of exit problems in this context that can be tackled using the ``scale functions paradigm'', and that is perhaps most comprehensive in the case of snLp. 

Our aim here is to provide a reference theory of scale functions and of the basic exit problems for pssMp with one-sided jumps, where as usual we exclude processes with monotone paths (up to absorption at the origin). More specifically, this paper will deliver:
\begin{enumerate}
\item an exposition of the scale functions for spectrally one-sided pssMp in a unified framework exposing their salient analytical features (Section~\ref{section:scales}); 
\item[and]
\item\label{intro:2} in terms of the latter:
\begin{enumerate}
\item\label{intro:a} the joint Laplace-Mellin transform of, in this order, the first exit time from a compact interval and position at this exit time for the pssMp, separately on the events corresponding to the continuous exit at the upper (resp. lower) boundary and the dis-continuous exit at the lower (resp. upper) boundary in the spectrally negative (resp. positive) case (Theorem~\ref{theorem:tse});
\item\label{intro:b} the joint Laplace-Laplace-Mellin transform-distribution function of, in this order, the first passage time of the multiplicative drawdown (resp. drawup) above a level that is a given function of the current maximum (resp. minimum), the time of the last maximum (resp. minimum) before, multiplicative overshoot of the drawdown (resp. drawup) at, and running supremum (infimum) at this first passage in the spectrally negative (resp. positive) case (Theorem~\ref{theorem:drawdown}\ref{drawdown:ii});
\item\label{intro:c} finally the  Laplace transform of the first passage time upwards (resp. downwards) for the pssMp on the event that this passage occurred before a multiplicative drawdown (resp. drawup) larger than a given function of the current maximum (resp. minimum) was seen (Theorem~\ref{theorem:drawdown}\ref{drawdown:i}). 
\end{enumerate}
\end{enumerate}

This programme was already initiated in \cite[Section~3.3]{zbigniew-bo}, where, however, only  the temporal two-sided exit problem in the spectrally positive case with no killing was considered. This paper then extends and complements those results by handling  the general completely asymmetric case, the joint spatio-temporal two-sided exit problem, as well as the drawdown (drawup) first passage problem, providing a systematic theory of scale functions for spectrally one-sided pssMp en route. 

In terms of existent related literature we must refer the reader also to \cite[Theorem~2.1]{pierre} \cite[Theorem~13.10]{kyprianou} for the first passage problem upwards of spectrally negative pssMp -- we touch on the latter only tangentially in this paper, leaving in fact open the parallel first passage downward problem for the spectrally positive case, see Question~\ref{question:downward}. Another query left open is given in

\begin{question}\label{question:drawup}
Can the problem of the first passage of the multiplicative drawup (resp. drawdown) in the spectrally negative (resp. positive) case similarly be solved (semi)explicitly in terms of the scale functions for pssMp, to be introduced presently?
\end{question}

The two main elements on which we will base our exposition are (i) the Lamperti transform \cite{lamperti} (see also \cite[Theorem 13.1]{kyprianou}) through which any pssMp can be expressed as the exponential of a (possibly killed) L\'evy process time-changed by the inverse of its exponential functional and (ii) the results of \cite{zbigniew-bo} concerning the two-sided exit problem of state-dependent killed snLp. It will emerge that even though pssMp, unlike processes with stationary independent increments, do not enjoy spatial homogeneity, nevertheless a  theory of scale functions almost entirely akin to that of snLp can be developed. In particular, because of the self-similarity property, it will turn out that one can still cope with scale functions depending on only one spatial variable (rather than two, which one would expect for a general spatially non-homogeneous Markov process with one-sided jumps). 
Of course the involvement of the time-change in the Lamperti transform means that some formulae end up being more involved than in the case of snLp and, for instance, appear to make Question~\ref{question:drawup} fundamentally more difficult than its snLp analogue. Nevertheless, the relative success of this programme begs 

\begin{question}\label{question:csbp}
Can a similar theory of scale functions be developed for continuous-state branching processes, which are another class of time-changed (and stopped), this time necessarily spectrally positive, L\'evy processes? (Again the transform is due to Lamperti \cite{lamperti-csbp}, see also \cite[Theorem~12.2]{kyprianou}.)
\end{question}
This too is left open to future work; we would point out only that a family of scale functions depending on a single spatial variable would probably no longer suffice, but that instead two would be needed. 

The organization of the remainder of this paper is as follows. We set the notation, recall the Lamperti transform and detail some further necessary tools in Section~\ref{section:preliminaries}. Then, using the results of \cite{zbigniew-bo}, we expound on the theory of scale functions for pssMp in Section~\ref{section:scales}. Sections~\ref{section:tse} and~\ref{section:drawdowns} contain, respectively, solutions to the two-sided exit and drawdown (drawup) first passage problems delineated above. Lastly, Section~\ref{section:conclusion} touches briefly on an application to a trailing stop-loss problem before closing. 

\section{Setting,  notation and preliminaries}\label{section:preliminaries}
Throughout we will write $\QQ[W]$ for $\EE_\QQ[W]$, $\QQ[W;A]$ for $\EE_\QQ[W\mathbbm{1}_A]$ and $\QQ[W\vert \mathcal{H}]$ for $\EE_\QQ[W\vert \mathcal{H}]$. More generally the integral $\int fd\mu$ will be written $\mu[f]$ etc. For $\sigma$-fields $\AA$ and $\BB$, $\AA/\BB$ will denote the set of $\AA/\BB$-measurable maps; $\mathcal{B}_A$ is the Borel (under the standard topology) $\sigma$-field on $A$.

We begin by taking $X=(X_t)_{t\in [0,\infty)}$, a snLp under the probabilities $(\PP_x)_{x\in \mathbb{R}}$ in the filtration $\FF=(\FF_t)_{t\in [0,\infty)}$. This means that $X$ is a  c\`adl\`ag, real-valued $\FF$-adapted process with stationary independent  increments relative to $\FF$, no positive jumps and non-monotone paths, that $\PP_0$-a.s. vanishes at zero; furthermore, for each $x \in \mathbb{R}$, the law of $X$ under $\PP_x$ is that of $x+X$ under $\PP_0$. We refer to \cite{bertoin,kyprianou,sato,doney} for the general background on (the fluctuation theory of)  L\'evy processes and to \cite[Chapter~VII]{bertoin} \cite[Chapter~8]{kyprianou} \cite[Chapter~9]{doney}  \cite[Section~9.46]{sato} for snLp in particular. As usual we set $\PP:=\PP_0$ and we assume $\FF$ is right-continuous. Expressions such as ``independent'', ``a.s.'', etc., without  further qualification, will mean ``independent (a.s., etc.) under $\PP_x$ for all $x\in \mathbb{R}$'', this having been of course also the case  for the occurence of ``stationary independent'' just above.

We let next $\expp$ be a strictly positive $\FF$-stopping time such that for some (then unique) $p\in [0,\infty)$, $\PP_x[g(X_{t+s}-X_t)\mathbbm{1}_{\{\expp>t+s\}}\vert\FF_t]=\PP[g(X_s)]e^{-p s}\mathbbm{1}_{\{\expp>t\}}$ a.s.-$\PP_x$ for all $x\in \mathbb{R}$, whenever $\{s,t\}\subset [0,\infty)$ and $g\in \mathcal{B}_\mathbb{R}/\mathcal{B}_{[0,\infty]}$;  in particular $\expp$ is exponentially distributed with rate $p$ ($\expp=\infty$ a.s. when $p=0$) independent of $X$.  Finally take an $\alpha\in \mathbb{R}$. 

We now associate to $X$, $\expp$ and $\alpha$ the process $Y=(Y_s)_{s\in [0,\infty)}$ as follows.  Set: $$I_t:=\int_0^{t }e^{\alpha X_u}\dd u,\quad t\in [0,\infty];$$
$$\phi_s:=\inf\{t\in [0,\infty):I_t>s\},\quad s\in [0,\infty);$$
finally $$Y_s:=
\begin{cases}
e^{X_{\phi_s}} & \text{for }s\in [0,\zeta)\\
\partial & \text{for }s\in [\zeta,\infty)
\end{cases},$$
where we consider $Y$ as having lifetime $\zeta:=I_\expp$ with $\partial\notin (0,\infty)$  the cemetery state. We take $\partial=0$ or $\partial=\infty$ according as $\alpha\geq 0$ or $\alpha< 0$ and set for convenience $\QQ_y:=\PP_{\log y}$ for $y\in (0,\infty)$ (naturally $\QQ:=\QQ_1$). 

When $\alpha>0$,  then $Y$ is nothing  but the pssMp associated to $X$, $\alpha$ and $\expp$ via the Lamperti transform. Likewise, when $\alpha<0$, then $1/Y$ ($1/\infty=0$) is the pssMp associated to $-X$, $-\alpha$ and $\expp$. Finally, when $\alpha=0$, then $Y$ is just the exponential of $X$ that has been killed at $\expp$ and sent to $-\infty$ ($e^{-\infty}=0$). Conversely, any positive pssMp with one-sided jumps and non-monotone paths up to absorption can be got in this way (possibly by enlarging the underlying probability space). For convenience we will refer to the association of $Y$ to $X$, as above, indiscriminately (i.e. irrespective of the sign of $\alpha$) as simply the Lamperti transform. 

Denote next by $\overline{Y}=(\overline{Y}_s)_{s\in [0,\zeta)}$ (resp. $\overline{X}=(\overline{X}_t)_{t\in [0,\infty)}$) the running supremum process of $Y$ (resp. $X$). We may then define the multiplicative (resp. additive) drawdown process/regret process/the process reflected multiplicatively (resp. additively) in its supremum $R=(R_s)_{s\in [0,\zeta)}$ (resp. $D=(D_t)_{t\in [0,\infty)}$) of $Y$ (resp. $X$) as follows: 
$$ R_s:=\frac{\overline{Y}_s}{Y_s}$$
for $s\in [0,\zeta)$ (resp. $$D_t:=\overline{X}_t-X_t$$ for $t\in [0,\infty)$). We also set ($Y_{0-}:=Y_0$) $$L_s:=\sup\{v\in [0,s]:\overline{Y}_v\in \{Y_v,Y_{v-}\}\},\quad s\in [0,\zeta),$$
and ($X_{0-}:=X_0$) $$G_t:=\sup\{u\in [0,t]:\overline{X}_u\in \{X_u,X_{u-}\}\},\quad t\in [0,\infty),$$
so that, for $s\in [0,\zeta)$ (resp. $t\in [0,\infty)$) $L_s$ (resp. $G_t$) is the last time $Y$ (resp. $X$) is at its running maximum on the interval $[0,s]$ (resp. $[0,t]$). Finally, for $c\in (0,\infty)$ and $r\in \mathcal{B}_{(0,\infty)}/\mathcal{B}_{(1,\infty)}$, respectively for $a\in \mathbb{R}$ and $s\in \mathcal{B}_{\mathbb{R}}/\mathcal{B}_{(0,\infty)}$, we introduce the random times $T_c^\pm$ and $\Sigma_r$, resp. $\tau_a^\pm$ and $\sigma_s$, by setting 
$$T_c^\pm:=\inf \{v\in (0,\zeta):\pm Y_v>\pm c\}\text{ and }\Sigma_r:=\inf\{v\in (0,\zeta):R_v>r(\overline{Y}_v)\},$$
resp. 
$$\tau_a^\pm:=\inf \{u\in (0,\infty):\pm X_u>\pm a\}\text{ and }\sigma_s:=\inf\{u\in (0,\infty):D_u>s(\overline{X}_u)\}$$
(with the usual convention $\inf\emptyset=\infty$). By  far the most important case for $r$, resp. $s$, is when this function is constant, in which case $\Sigma_r$, resp. $\sigma_s$, becomes the first passage time upwards for $R$, resp. $D$. We keep the added generality since it is with essentially no cost to the complexity of the results, and may prove valuable in applications.

\begin{remark}\label{rmk:dual}
The case when, ceteris paribus, $X$ is spectrally positive rather than spectrally negative, may be handled by applying the results to $X':=-X$ in place of $X$ and $\alpha':=-\alpha$ in place of $\alpha$: if $Y'$ corresponds to $(X',\alpha',\expp)$ as $Y$ corresponds to $(X,\alpha,\expp)$, then $Y'=1/Y$, the running infimum process of $Y$ is equal to $\underline{Y'}=1/\overline{Y}$ and the multiplicative drawup of $Y'$ is $\hat{R}':=\frac{Y'}{\underline{Y'}}=R$. Hence, viz. Item~\eqref{intro:2} from the Introduction, our results will apply (modulo trivial spatial transformations) also in the case when $X$ is spectrally positive and we have lost, thanks to allowing $\alpha$ to be an arbitrary real number (so not necessarily positive), no generality in assuming that $X$ is spectrally negative, rather than merely completely asymmetric.
\end{remark}

Some further notation. As usual we will denote by $\psi$ the Laplace exponent of $X$, $\psi(\lambda):=\log \PP[e^{\lambda X_1}]$ for $\lambda\in [0,\infty)$. It has the representation 
\begin{equation}\label{eq:laplace-exp}
\psi(\lambda)=\frac{\diffusion}{2}\lambda^2+\mu\lambda+\int_{(-\infty,0)}(e^{\lambda y}-\mathbbm{1}_{[-1,0)}(y)\lambda y-1)\nu(dy),\quad \lambda\in [0,\infty),
\end{equation}
for some (unique) $\mu\in \mathbb{R}$, $\diffusion\in [0,\infty)$, and measure $\nu$ on $\mathcal{B}_\mathbb{R}$, supported by $(-\infty,0)$, and satisfying $\int (1\land y^2 )\nu(\dd y)<\infty$. When $X$ has paths of finite variation, equivalently $\diffusion=0$ and $\int (1\land \vert y\vert)\nu(dy)<\infty$, we set $\drift:=\mu+\int_{[-1,0)}\vert y\vert\nu(dy)$; in this case we must have $\drift\in (0,\infty)$ and $\nu$ non-zero. For convenience we  interpret $\drift=\infty$ when $X$ has paths of infinite variation. We also put $\Phi:=(\psi\vert_{[\Phi(0),\infty)})^{-1}$ (the inverse function is meant), where $\Phi(0)$ is the largest zero of $\psi$; and introduce the shorthand notation (cf. \cite[Eq.~(13.54)]{kyprianou})
\begin{equation}\label{eq:shorthand}
\mathsf{a}_{\alpha,p}(\lambda,k):=\left(\prod_{l=0}^k(\psi(\lambda-l\alpha)-p)\right)^{-1},\quad k\in \mathbb{N}_0,
\end{equation}
whenever it is defined.

We recall now two tools from the fluctuation theory of snLp that will prove useful later on. 

The first of these is the vehicle of the Esscher transform \cite[Eq.~(8.5)]{kyprianou}. To wit, under certain technical conditions on the underlying filtered space that we may assume hold without loss of generality when proving distributional results, for any $\theta\in [0,\infty)$ there exists a unique family of measures $(\PP^\theta_x)_{x\in \mathbb{R}}$ on $\FF_\infty$ such that \cite[Corollary~3.11]{kyprianou} for any $\FF$-stopping time $T$, 
$$\frac{\dd\PP^{\theta}_x\vert_{\FF_T}}{\dd\PP_x\vert_{\FF_T}}=e^{\theta(X_T-X_0)-\psi(\theta)T}\text{ a.s.-}\PP_x\text{ on }\{T<\infty\},\quad x\in \mathbb{R};$$
 of course we set $\PP^\theta:=\PP^\theta_0$.

The second tool is It\^o's \cite{ito,blumenthal} Poisson point process (Ppp) of excursions of $X$ from its maximum. Let $\mathbb{D}$ be the space of c\`adl\`ag real-valued paths on $[0,\infty)$, let $\Delta\notin \mathbb{D}$ be a coffin state, and set $\chi(\omega):=\inf\{u\in  (0, \infty):\omega(u)\geq  0\}$ for $\omega\in \mathbb{D}$. Then the reader will recall that under $\PP$ the running supremum $\overline{X}$ serves as a continuous local time for $X$ at the maximum and that, moreover, the process $\epsilon=(\epsilon_g)_{g\in (0,\infty)}$, defined as follows: 
\begin{quote}
 for $g\in (0,\infty)$, $\epsilon_g(u):=X_{(\tau_{g-}^++u)\land \tau_g^+}-X_{\tau_{g-}^+-}$ for $u\in [0,\infty)$ (of course $X_{\tau_{g-}^+-}=g$ off the event $\{X_0\ne 0\}$, which is $\PP$-negligible) if $g\in \DD:=\{f\in (0,\infty):\tau_{f-}^+<\tau_f^+\}$; $\epsilon_g:=\Delta$ otherwise, 
\end{quote}
is, under $\PP$, a  $\mathbb{D}$-valued Ppp in the filtration $(\FF_{\tau^+_a})_{a\in [0,\infty)}$, absorbed on first entry into a path for which $\chi=\infty$, and whose characteristic measure we will denote by $\nn$ (so the intensity measure of $\epsilon$ is $\mathsf{l}\times \nn$, where $\mathsf{l}$ is Lebesgue measure on $\mathcal{B}_{(0,\infty)}$) \cite{greenwood-pitman,rogers}.  $\xi$ will denote the coordinate process on $\mathbb{D}$, $\underline{\xi}_\infty$ will be its overall infimum, and, for $s\in \mathbb{R}$, $S_s^\pm$ will be the first hitting time of the set $\pm (s,\infty)$ by the process $\pm \xi$. 

Finally, in terms of general notation, we have as follows. $\star$ will denote  convolution on the real line: for $\{f,g\}\subset \mathcal{B}_{\mathbb{R}}/\mathcal{B}_{[-\infty,\infty]}$, $$(f\star g)(x):=\int_{-\infty}^\infty f(y)g(x-y)\dd y,\quad x\in \mathbb{R},$$ whenever the Lebesgue integral is well-defined. For a function $f\in \mathcal{B}_{\mathbb{R}}/\mathcal{B}_{[0,\infty)}$ vanishing on $(-\infty,0)$, $\hat{f}:[0,\infty)\to [0,\infty]$ will be its Laplace transform, $$\hat{f}(\lambda):=\int_0^\infty e^{-\lambda x}f(x)\dd x,\quad \lambda\in [0,\infty);$$ sometimes we will write $f^\wedge$ in place of $\hat{f}$ for typographical ease.  Given an expression $\mathcal{R}(x)$ defined for $x\in R$ we will write $\mathcal{R}(\cdot)$ for the function $(R\ni x\mapsto \mathcal{R}(x))$, with $R$ being understood from context. We interpret $a/\infty=0$ for $a\in [0,\infty)$.

\section{Scale functions for pssMp}\label{section:scales}
Associated to $X$ is a family  $(\Wq)_{q\in [0,\infty)}$ of so-called scale functions that feature heavily in first passage/exit and related fluctuation identities. Specifically, for $q\in [0,\infty)$, $\Wq$ is characterized as the unique function mapping $\mathbb{R}$ to $[0,\infty)$, vanishing on $(-\infty,0)$, continuous on $[0,\infty)$, and having Laplace transform $$\widehat{\Wq}(\theta)=\frac{1}{\psi(\theta)-q},\quad \theta\in (\Phi(q),\infty).$$ As usual we set $W:=W^{(0)}$ for \emph{the} scale function of $X$. We refer to \cite{kkr} for an overview of the general theory of scale functions of snLp, recalling here explicitly only  the relation \cite[Eq.~(25)]{kkr}  
\begin{equation}\label{eq:Esscher-W}
W^{(p)}=e^{\Phi(p)\cdot}W_{\Phi(p)},
\end{equation}
where $W_{\Phi(p)}$ is the scale function for $X$ under $\PP^{\Phi(p)}$, and the following estimate of the proof of \cite[Lemma~3.6, Eq.~(55)]{kkr}: 
\begin{equation} \label{eq:estimate}
(\star^{k+1}W)(x)\leq \frac{x^k}{k!}W(x)^{k+1},\quad x\in \mathbb{R},\, k\in \mathbb{N}_0,
\end{equation}
where $\star^{k+1}W$ is the $(k+1)$-fold convolution of $W$ with itself.

It will emerge that the correct (from the point of view of fluctuation theory) analogues of these functions in the setting of pssMp are contained in
\begin{definition}
For each $q\in [0,\infty)$, let  $\mathcal{W}^{(q)}_{\alpha,p}=\mathcal{W}\circ \log$ for \cite[Lemma~2.1]{zbigniew-bo} the unique locally bounded and Borel measurable function $\mathcal{W}:\mathbb{R}\to \mathbb{R}$ such that 
\begin{equation}\label{eq:convolution:W}
\mathcal{W}=\Wp+q\Wp\star (e^{\alpha\cdot} \mathbbm{1}_{[0,\infty)}\mathcal{W}).
\end{equation}
\end{definition}
\begin{remark}
By definition $\mathcal{W}^{(q)}_{\alpha,p}$ vanishes on $(0,1)$, while $\mathcal{W}^{(q)}_{\alpha,p}(1)=\Wp(0)$.
\end{remark}
\begin{remark}\label{remark:equivalent-W}
The convolution equation \eqref{eq:convolution:W} in the definition is equivalent  \cite[Lemma~2.1 \& Eq.~(2.10)]{zbigniew-bo} to each of 
$$\mathcal{W}=W+W\star ((qe^{\alpha \cdot}+p)\mathbbm{1}_{[0,\infty)}\mathcal{W})$$ and  $$\mathcal{W}=W^{(p+q)}+qW^{(p+q)}\star ((e^{\alpha \cdot}-1)\mathbbm{1}_{[0,\infty)}\mathcal{W}).$$
\end{remark}
\begin{remark}\label{remark:special}
Two special cases: (i) when $\alpha=0$, then $\mathcal{W}^{(q)}_{\alpha,p}=W^{(q+p)}\circ \log$ and (ii) when $q=0$, then  $\mathcal{W}^{(q)}_{\alpha,p}=\Wp\circ \log$.
\end{remark}

\begin{proposition}\label{proposition:Esscher:W}
Let $q\in[0,\infty)$. For $\theta\in [0,\Phi(p)]$ one has  $\mathcal{W}_{\alpha,p}^{(q)}(y)={_\theta\mathcal{W}}_{\alpha,p-\psi(\theta)}^{(q)}(y)y^\theta $, $y\in (0,\infty)$, where the left subscript $\theta$ indicates that the quantity is to be computed for the process $X$ under $\PP^\theta$. 
\end{proposition}

\begin{proof}
By Remark~\ref{remark:equivalent-W},  ${_\theta\mathcal{W}}_{\alpha,p-\psi(\theta)}^{(q)}=\mathcal{W}\circ \log$, where $\mathcal{W}$ is the unique locally bounded and Borel measurable solution to 
$$\mathcal{W}=W^{(q+p-\psi(\theta))}_\theta+q W^{(q+p-\psi(\theta))}_\theta\star ((e^{\alpha \cdot}-1)\mathbbm{1}_{[0,\infty)}\mathcal{W}).$$
(Again the subscript $\theta$ indicates that the quantity is to be computed for $X$ under $\PP^\theta$.) In other words, because  \cite[Eq.~(8.30)]{kyprianou} $W^{(q+p)}=e^{\theta\cdot} W^{(q+p-\psi(\theta))}_\theta$, it is seen, still by Remark~\ref{remark:equivalent-W}, that
$e^{\theta\cdot}\mathcal{W}$ satisfies the convolution equation that characterizes $\mathcal{W}^{(q)}_{\alpha,p}\circ e^\cdot$.
\end{proof}
Next, just as in the case of snLp, there is a representation of the scale functions in terms of the excursion measure $\nn$.

\begin{lemma}\label{lemma:fundamental}
For $r\in [1,\infty)$, $q\in [0,\infty)$ and then $c\in [r,\infty)$, $c>1$, one has
$$-\log \frac{\WW(r)}{\WW(c)}=\frac{p\log\frac{c}{r}+\frac{q}{\alpha} (c^\alpha-r^\alpha)}{\drift}+\int_{\log(r)}^{\log(c)}\nn[1-e^{-\int_0^\chi(qe^{\alpha g}e^{\alpha\xi_t}+p)\dd t}\mathbbm{1}_{\{-\underline{\xi}_\infty\leq g\}}]\dd g,$$
where the first expression appears only when $X$ has paths of finite variation and it must then further be understood in the limiting sense when $\alpha=0$.
\end{lemma}
The proof of this lemma is deferred to the next section (p.~\pageref{proof-conclusion}) where it will, of course independently,  fall  out naturally from the study of the two-sided exit problem.

The following proposition now gathers some basic analytical properties of the system $(\mathcal{W}^{(q)}_{\alpha,p})_{q\in [0,\infty)}$.
\begin{proposition}\label{proposition:analytic}
For each $q\in [0,\infty)$:
\begin{enumerate}
\item\label{analytic:a} One has $\mathcal{W}^{(q)}_{\alpha,p}= \mathcal{W}^{(q)}_{\alpha,p,\infty}\circ \log$, where $ \mathcal{W}^{(q)}_{\alpha,p,\infty}:=\uparrow \!\!\!\!\text{-}\lim_{n\to\infty} \mathcal{W}^{(q)}_{\alpha,p,n}$, with $\mathcal{W}^{(q)}_{\alpha,p,0}:=\Wp$ and then inductively $\mathcal{W}^{(q)}_{\alpha,p,n+1}:=\Wp+q\Wp\star (e^{\alpha \cdot}\mathcal{W}^{(q)}_{\alpha,p,n})$ for $n\in \mathbb{N}_0$, in other words $\mathcal{W}^{(q)}_{\alpha,p,\infty}=\sum_{k=0}^{\infty} q^k [\star_{l=0}^k(\Wp e^{l\alpha \cdot})]$.
\item\label{analytic:b} For $\alpha\leq 0$, $\WW\leq W^{(p+q)}\circ \log$ and  $\widehat{\mathcal{W}^{(q)}_{\alpha,p,\infty}}=\sum_{k=0}^\infty \mathsf{a}_{\alpha,p}(\cdot,k)q^{k}<\infty$ on $(\Phi(p+q),\infty)$. 
\item\label{analytic:c} $\mathcal{W}^{(q)}_{\alpha,p}$ is continuous strictly increasing on $[1,\infty)$, and $\mathcal{W}^{(q)}_{\alpha,p}\vert_{(1,\infty)}$ admits a locally bounded left-continuous left- and a locally bounded right-continuous right- derivative that coincide everywhere except at most on a countable set: in fact they coincide everywhere when $X$ has paths of infinite variation and otherwise they agree off $\{x\in (0,\infty):\text{the L\'evy measure of }X\text{ has positive mass at }-x\}$.  The left and right derivative can be made explicit: for $r\in (1,\infty)$,
$$\frac{r(\mathcal{W}^{(q)}_{\alpha,p})_+^\prime(r)}{\mathcal{W}^{(q)}_{\alpha,p}(r)}=\frac{p+qr^\alpha}{\delta}+\nn\left[1-\exp\left(-\int_0^\chi (qr^\alpha e^{\alpha\xi_t}+p)\dd t\right)\mathbbm{1}_{\{-\underline{\xi}_\infty\leq \log(r)\}}\right],$$
while 
$$\frac{r(\mathcal{W}^{(q)}_{\alpha,p})_-^\prime(r)}{\mathcal{W}^{(q)}_{\alpha,p}(r)}=\frac{p+qr^\alpha}{\delta}+\nn\left[1-\exp\left(-\int_0^\chi (qr^\alpha e^{\alpha\xi_t}+p)\dd t\right)\mathbbm{1}_{\{-\underline{\xi}_\infty< \log(r)\}}\right]$$
(recall we interpret $\delta=\infty$ and hence $\frac{p+qr^\alpha}{\delta}=0$ when $X$ has paths of infinite variation).
\item\label{analytic:0+}  $\mathcal{W}^{(q)}_{\alpha,p,\infty}(y)\sim \Wp(y)\sim W(y)$ as $y\downarrow 0$. When $\alpha< 0$ and $q>0$, then $\WW(y) y^{-\Phi(p+q)}\sim \frac{\left(\sum_{k=0}^\infty \mathsf{a}_{\alpha,p}(\Phi(p+q),k)q^k\right)^2}{\sum_{k=0}^\infty \mathsf{a}_{\alpha,p}(\Phi(p+q),k)\left(\sum_{l=0}^k\frac{\psi'(\Phi(p+q)-l\alpha)}{\psi(\Phi(p+q)-l\alpha)-p}\right)q^k}\in (0,\infty)$ as $y\to \infty$.
 \end{enumerate}
Besides, $([0,\infty)\times [1,\infty)\ni (q,y)\mapsto \WW(y))$ is continuous. Finally, for each $y\in (1,\infty)$:
\begin{enumerate}[resume]
 \item\label{analytic:d} The function $([0,\infty)\ni q\mapsto \mathcal{W}^{(q)}_{\alpha,p}(y))$ extends to an entire function. 
\item\label{analytic:e} If $\Wp$ is continuously differentiable on $(0,\infty)$ (see \cite[Lemma~2.4]{kkr} for equivalent conditions), then for each $q\in [0,\infty)$ so is $\mathcal{W}^{(q)}_{\alpha,p}$, and $([0,\infty)\ni q\mapsto \mathcal{W}^{(q)\prime}_{\alpha,p}(y))$ extends to an entire function.
\end{enumerate}
\end{proposition}
\begin{remark}
For $\alpha>0$, $\widehat{\mathcal{W}^{(q)}_{\alpha,p,\infty}}$ is finite on no neighborhood of $\infty$. Indeed, from \eqref{analytic:a}, irrespective of the sign of $\alpha$, $\WW$ expands into a nonnegative function power-series in $q$, the function-coefficients of which have the following Laplace transforms: for $k\in \mathbb{N}_0$, $(\star_{l=0}^k(\Wp e^{l\alpha \cdot}))^\wedge=\mathsf{a}_{\alpha,p}(\cdot,k)<\infty$ on $(\Phi(p)+k(\alpha\lor 0),\infty)$ and $(\star_{l=0}^k(\Wp e^{l\alpha \cdot}))^\wedge=\infty$ off this set. In particular it is certainly not the case that the asymptotics at $\infty$ of \eqref{analytic:0+} extends to the case $\alpha\geq 0$ (though, for $\alpha=0$ or $q=0$, the asymptotics is that of the scale functions of $X$, viz. Remark~\ref{remark:special}).
\end{remark}
\begin{remark}
From the series representation in \eqref{analytic:a} it is clear that the computation of $\WW$ boils down to algebraic  manipulations whenever $\Wp$ is a mixture of exponentials. This is for instance the case for Brownian motion with drift and exponential jumps (except for special constellations of the parameters) \cite[Eq.~(7)]{kkr}. 
\end{remark}
\begin{proof}
\eqref{analytic:a}. By monotone convergence $\mathcal{W}^{(q)}_{\alpha,p,\infty}$ satisfies the defining convolution equation \eqref{eq:convolution:W}. Furthermore, it is proved by induction that $\mathcal{W}^{(q)}_{\alpha,p,n}\leq \mathcal{W}^{(q)}_{\alpha,p}\circ e^\cdot$ for all $n\in\mathbb{N}_0$, and hence $0\leq \mathcal{W}^{(q)}_{\alpha,p,\infty}\leq \mathcal{W}^{(q)}_{\alpha,p}\circ e^\cdot$. In consequence $\mathcal{W}^{(q)}_{\alpha,p,\infty}$ is locally bounded and Borel measurable, so that by the uniqueness of the solution to \eqref{eq:convolution:W} we must have $\mathcal{W}^{(q)}_{\alpha,p}=\mathcal{W}^{(q)}_{\alpha,p,\infty}\circ\log$. 

\eqref{analytic:b}. When $\alpha\leq0$, then $\mathcal{W}^{(q)}_{\alpha,n+1}=\Wp+q\Wp\star (e^{\alpha\cdot}\mathcal{W}^{(q)}_{\alpha,p,n})$ implies $\widehat{\mathcal{W}^{(q)}_{\alpha,p,n+1}}=(\psi-p)^{-1}(1+q\widehat{\mathcal{W}^{(q)}_{\alpha,p,n}}(\cdot-\alpha))$ on $(\Phi(p),\infty)$, which together with $\widehat{\mathcal{W}^{(q)}_{\alpha,p,0}}=(\psi-p)^{-1}$ and $\widehat{\mathcal{W}^{(q)}_{\alpha,p,\infty}}=\lim_{n\to\infty}\widehat{\mathcal{W}^{(q)}_{\alpha,p,n}}$ (the latter being a consequence of monotone convergence) entails $\widehat{\mathcal{W}^{(q)}_{\alpha,p,\infty}}=\sum_{k=0}^\infty \mathsf{a}_{\alpha,p}(\cdot,k)q^{k}$ on $(\Phi(p),\infty)$. The finiteness comes from the observation that $\WW\leq W^{(p+q)}\circ \log$, which is an immediate consequence of Remark~\ref{remark:equivalent-W}. 

\eqref{analytic:c}. By an expansion into a $q$-series from \eqref{analytic:a} and from the corresponding properties of $\Wp$, clearly $\mathcal{W}^{(q)}_{\alpha,p,\infty}$ is strictly increasing, while dominated convergence using \eqref{eq:Esscher-W}-\eqref{eq:estimate} implies that $\mathcal{W}^{(q)}_{\alpha,p,\infty}$ is continuous on $[0,\infty)$ (the same argument gives the joint continuity from the statement immediately following \eqref{analytic:c}). The final part of this statement follows from Lemma~\ref{lemma:fundamental} and from the considerations of the proof of \cite[Lemma~1(iii)]{doney-conditioning}.

In a similar vein, by dominated convergence for the series in $q$, the first part of \eqref{analytic:0+} is got (the asymptotic equivalence of $\Wp$ and $W$ at $0+$ follows itself from  \cite[Lemma~3.6, Eqs.~(55) \& (56)]{kkr} and dominated convergence). 

Furthermore, when $\alpha< 0$ and $q>0$, it follows from Remark~\ref{remark:equivalent-W}, monotone convergence and from the boundedness and monotonicity of $W^{(p+q)}e^{-\Phi(p+q)\cdot}=W_{\Phi(p+q)}$ that $\lim_\infty\mathcal{W}^{(q)}_{\alpha,p,\infty} e^{-\Phi(p+q)\cdot}$ exists in $[0,\infty)$, in which case \eqref{analytic:b}  renders $\lim_\infty \mathcal{W}^{(q)}_{\alpha,p,\infty} e^{-\Phi(p+q)\cdot}=\lim_{\lambda\downarrow 0}\lambda \sum_{k=0}^\infty \mathsf{a}_{\alpha,p}(\Phi(p+q)+\lambda,k)q^{k}$ (we have used here the following: assuming $f\in \mathcal{B}_\mathbb{R}/\mathcal{B}_{[0,\infty)}$, vanishing on $(-\infty,0)$, is bounded and $\lim_\infty f$ exists, then one may write, by bounded convergence, $\lambda\hat{f}=\mathsf{R}[f(e_1/\lambda)]\to \lim_\infty f$ as $\lambda\downarrow 0$ for $e_1\sim_\mathsf{R}\mathrm{Exp}(1)$). But  $$\lim_{\lambda\downarrow 0}\lambda \sum_{k=0}^\infty \mathsf{a}_{\alpha,p}(\Phi(p+q)+\lambda,k)q^{k}=\frac{\left(\sum_{k=0}^\infty \mathsf{a}_{\alpha,p}(\Phi(p+q),k)q^k\right)^2}{\sum_{k=0}^\infty \mathsf{a}_{\alpha,p}(\Phi(p+q),k)\left(\sum_{l=0}^k\frac{\psi'(\Phi(p+q)-l\alpha)}{\psi(\Phi(p+q)-l\alpha)-p}\right)q^k}\in (0,\infty),$$ by l'H\^ospital's rule: indeed since (as is easy to check) $\psi'/\psi$ is bounded on $[c,\infty)$ for any $c\in (\Phi(0),\infty)$ and since $\psi$ grows ultimately at least linearly, differentiation under the summation sign can be justified by the fact that the resulting series converges absolutely locally uniformly in $\lambda\in (0,\infty)$, and then the limit as $\lambda\downarrow 0$ can be taken via dominated convergence.

\eqref{analytic:d} and \eqref{analytic:e}. The series in $q$ got in \eqref{analytic:a} converges for all $q\in [0,\infty)$ and it has nonnegative coefficients; it is immediate that it extends to an entire function. Furthermore, if $\Wp$ is of class $C^1$ on $(0,\infty)$, then this series may be differentiated term-by-term, because -- differentiation under the sum -- the resulting differentiated series can be dominated locally uniformly in $q\in [0,\infty)$ by a summable series on account of \eqref{eq:Esscher-W}-\eqref{eq:estimate}.  In particular the resulting derivative is continuous and extends to an entire function in $q$. (In fact some non-trivial care is needed even for the differentiation of the individual terms in the series from  \eqref{analytic:a}: they are, to be differentiated in $x\in (0,\infty)$, of the form $\int_0^xg(y)\Wp(x-y)dy$ for a continuous $g:[0,\infty)\to [0,\infty)$. It is then an exercise in real analysis to convince oneself that the usual Leibniz' rule for differentiation under the integral sign/in the delimiters applies here. A useful property to be used to this end is that the derivative of $W^{(p)}$ on $(0,\infty)$ decomposes into a continuous part that is bounded on each bounded interval and  a (therefore necessarily integrable on each bounded interval) continuous nonincreasing part. The latter  is seen to hold true from \eqref{analytic:c}. We omit further details.)
\end{proof}
Besides $(\Wq)_{q\in [0,\infty)}$ one has in the theory of snLp the ``adjoint'' scale functions $(Z^{(q,\theta)})_{(q,\theta)\in [0,\infty)^2}$, which are also very convenient in organizing fluctuation results. To wit, for $\{q,\theta\}\subset [0,\infty)$, \cite[Eq.~(5.17)]{app}
$$\Zqt:=e^{\theta \cdot}+(q-\psi(\theta))(e^{\theta\cdot}\mathbbm{1}_{[0,\infty)})\star \Wq.$$
In particular we write $Z^{(q)}:=Z^{(q,0)}$, $q\in [0,\infty)$. It is easy to check the following for $\{q,\theta\}\subset [0,\infty)$: one has the Laplace transform
$$(\Zqt\mathbbm{1}_{[0,\infty)})^\wedge (\lambda)=\frac{\psi(\lambda)-\psi(\theta)}{(\lambda-\theta)(\psi(\lambda)-q)},\quad \lambda\in (\Phi(q)\lor\theta,\infty);$$  $\Zqt$ is $(0,\infty)$-valued;  $\Zqt(x)=e^{\theta x}$ for $x\in (-\infty,0]$.

The analogues of these functions in the context of pssMp are given by

\begin{definition}
For $\{q,\theta\}\subset [0,\infty)$, let   $\mathcal{Z}^{(q,\theta)}_{\alpha,p}:=\mathcal{Z}\circ \log$ for \cite[Lemma~2.1]{zbigniew-bo} the unique locally bounded Borel measurable  $\mathcal{Z}:\mathbb{R}\to \mathbb{R}$ satisfying the convolution equation 
\begin{equation}\label{eq:convolution:Z}
\mathcal{Z}=Z^{(p,\theta)}+q\Wp\star(e^{\alpha\cdot} \mathbbm{1}_{[0,\infty)}\mathcal{Z}).
\end{equation}
We write $\ZZq:=\mathcal{Z}^{(q,0)}_{\alpha,p}$ for short.
\end{definition}
\begin{remark}
The definition entails that $\mathcal{Z}^{(q,\theta)}_{\alpha,p}(y)=y^\theta$ for $y\in (0,1]$.
\end{remark}

\begin{remark}\label{remark:equivalent-Z}
Exploiting the relation $Z^{(p,\theta)}=Z^{(0,\theta)}+pW^{(p)}\star Z^{(0,\theta)}$, that may be checked via Laplace transforms, the convolution equation \eqref{eq:convolution:Z} is seen to be equivalent  \cite[Lemma~2.1]{zbigniew-bo} to each of
$$\mathcal{Z}=Z^{(0,\theta)}+W\star((qe^{\alpha\cdot} +p) \mathbbm{1}_{[0,\infty)}\mathcal{Z})$$
and
$$\mathcal{Z}=Z^{(q+p,\theta)}+qW^{(q+p)}\star((e^{\alpha\cdot} -1) \mathbbm{1}_{[0,\infty)}\mathcal{Z}).$$
\end{remark}

\begin{remark}
Two special cases: when (i) $\alpha=0$, then $\mathcal{Z}^{(q,\theta)}_{\alpha,p}=Z^{(q+p,\theta)}\circ \log$ and (ii) when $q=0$, then $\mathcal{Z}^{(q,\theta)}_{\alpha,p}=Z^{(p,\theta)}\circ \log$.
\end{remark}
Parallel to Proposition~\ref{proposition:Esscher:W} we have 
\begin{proposition}\label{proposition:Esscher:Z}
For $\theta\in [0,\Phi(p)]$, $q\in[0,\infty)$, one has $\mathcal{Z}_{\alpha,p}^{(q,\theta)}(y)={_\theta\mathcal{Z}}_{\alpha,p-\psi(\theta)}^{(q)}(y)y^\theta$, $y\in (0,\infty)$, where the left subscript $\theta$ indicates that the quantity is to be computed for the process $X$ under $\PP^\theta$. 
\end{proposition}
\begin{proof}
The proof is essentially verbatim that of Proposition~\ref{proposition:Esscher:W}, except that now one exploits the identity $Z^{(q+p,\theta)}=e^{\theta\cdot} Z^{(q+p-\psi(\theta))}_\theta$ and Remark~\ref{remark:equivalent-Z}.
\end{proof}
We establish also the basic analytical properties of the family $(\mathcal{Z}_{\alpha,p}^{(q,\theta)})_{(q,\theta)\in [0,\infty)^2}$.
\begin{proposition}\label{proposition:analytic:Z}
For $\{q,\theta\}\subset [0,\infty)$:
\begin{enumerate}
\item\label{analytic:Z:recursion} One has $\mathcal{Z}^{(q,\theta)}_{\alpha,p}= \mathcal{Z}^{(q,\theta)}_{\alpha,p,\infty}\circ \log$, where $ \mathcal{Z}^{(q,\theta)}_{\alpha,p,\infty}:=\uparrow \!\!\!\!\text{-}\lim_{n\to\infty}  \mathcal{Z}^{(q,\theta)}_{\alpha,p,n}$, with $ \mathcal{Z}^{(q,\theta)}_{\alpha,p,0}:=Z^{(p,\theta)}$ and then inductively $ \mathcal{Z}^{(q,\theta)}_{\alpha,p,n+1}:=Z^{(p,\theta)}+q\Wp\star (e^{\alpha \cdot} \mathcal{Z}^{(q,\theta)}_{\alpha,p,n}\mathbbm{1}_{[0,\infty)})$ for $n\in \mathbb{N}_0$, in other words $ \mathcal{Z}^{(q,\theta)}_{\alpha,p,\infty}=Z^{(p,\theta)}+\sum_{k=1}^\infty q^k[\star_{l=0}^{k-1}(\Wp e^{l\alpha\cdot })]\star (e^{k\alpha  \cdot}Z^{(p,\theta)}\mathbbm{1}_{[0,\infty)})$.
\item\label{analytic:Z:laplace} When $\alpha\leq 0$, then $\ZZ\leq Z^{(p+q,\theta)}\circ\log$ and $(\mathcal{Z}^{(q,\theta)}_{\alpha,p,\infty}\mathbbm{1}_{[0,\infty)})^\wedge=\sum_{k=0}^\infty \mathsf{a}_{\alpha,p}(\cdot,k)\frac{\psi(\cdot-k\alpha)-\psi(\theta)}{(\cdot-k\alpha-\theta)}q^k<\infty$ on $(\Phi(p+q)\lor \theta,\infty)$.
\item\label{analytic:Z:ct-diif} $\mathcal{Z}^{(q,\theta)}_{\alpha,p}$ is continuous and $\mathcal{Z}^{(q,\theta)}_{\alpha,p}\vert_{[1,\infty)}$ is continuously differentiable. If $W^{(p)}$ is of class $C^1$ on $(0,\infty)$ then $\mathcal{Z}^{(q,\theta)}_{\alpha,p}\vert_{(1,\infty)}$ is twice continuously differentiable.
 \end{enumerate}
Besides, for each $\theta\in [0,\infty)$, $([0,\infty)\times [1,\infty)\ni (q,y)\mapsto \ZZ(y))$ is continuous. Finally, for each $y\in [1,\infty)$ and $\theta\in [0,\infty)$: 
\begin{enumerate}[resume]
 \item\label{analytic:Z:entire} The functions $([0,\infty)\ni q\mapsto \mathcal{Z}^{(q,\theta)}_{\alpha,p}(y))$ and $([0,\infty)\ni q\mapsto \mathcal{Z}^{(q,\theta)\prime}_{\alpha,p}(y))$ extend to entire functions. 
\end{enumerate}
\end{proposition}
\begin{remark}
For $\alpha>0$,  $(\mathcal{Z}^{(q,\theta)}_{\alpha,p,\infty}\mathbbm{1}_{[0,\infty)})^\wedge$ is finite on no neighborhood of $\infty$. Indeed, from \eqref{analytic:Z:recursion}, irrespective of the sign of $\alpha$, $\ZZ$ expands into a nonnegative function power-series in $q$, the function-coefficients of which have the Laplace transforms given as follows: for $k\in \mathbb{N}_0$, $((\star_{l=0}^{k-1}\Wp e^{l\alpha\cdot })\star (e^{k\alpha  \cdot}Z^{(p,\theta)}\mathbbm{1}_{[0,\infty)}))^\wedge=\mathsf{a}_{\alpha,p}(\cdot,k)\frac{\psi(\cdot-k\alpha)-\psi(\theta)}{(\cdot-k\alpha-\theta)}<\infty$ on $(\Phi(p)\lor\theta+k(\alpha\lor 0),\infty)$, while $((\star_{l=0}^{k-1}\Wp e^{l\alpha\cdot })\star (e^{k\alpha  \cdot}Z^{(p,\theta)}\mathbbm{1}_{[0,\infty)}))^\wedge=\infty$ off this set.
\end{remark}
\begin{remark}
Combining Proposition~\ref{proposition:analytic}\eqref{analytic:a} and Proposition~\ref{proposition:analytic:Z}\eqref{analytic:Z:recursion} we see that $\ZZ=Z^{(p,\theta)}+q\sum_{k=0}^\infty q^k \mathcal{W}_{\alpha,p}^{(q)_k}\star [e^{\alpha(k+1)\cdot} Z^{(p,\theta)}\mathbbm{1}_{[0,\infty)}]$, where for $n\in \mathbb{N}_0$, $\mathcal{W}_{\alpha,p}^{(q)_n}$ is the function-coefficient at $q^n$ in the expansion of $\WW$ into a $q$-power series. When $\alpha=0$ (but not otherwise) this of course simplifies to $\ZZ=Z^{(p,\theta)}+q\WW\star (Z^{(p,\theta)}\mathbbm{1}_{[0,\infty)})$. In any event, however, in terms of numerics,  if $\WW$ has been determined by expansion into a $q$-series, then the function-coefficients of the expansion of $\ZZ$ into a $q$-series are ``only another convolution away''. These are only very superficial comments, though, and an investigation of the numerical evaluation of  scale functions for spectrally one-sided pssMp is left to be pursued elsewhere.
\end{remark} 
\begin{proof}
The proof of \eqref{analytic:Z:recursion} is essentially verbatim that of Proposition~\ref{proposition:analytic}\eqref{analytic:a}.

\eqref{analytic:Z:laplace}. When $\alpha\leq 0$, taking Laplace transforms in $\mathcal{Z}^{(q,\theta)}_{\alpha,p,n+1}=Z^{(p,\theta)}+q\Wp\star(e^{\alpha\cdot} \mathbbm{1}_{[0,\infty)}\mathcal{Z}^{(q,\theta)}_{\alpha,p,n})$ yields $(\mathcal{Z}^{(q,\theta)}_{\alpha,p,n+1}\mathbbm{1}_{[0,\infty)})^\wedge(\lambda)=\frac{\psi(\lambda)-\psi(\theta)}{(\lambda-\theta)(\psi(\lambda)-p)}+\frac{q}{\psi(\lambda)-p}(\mathcal{Z}^{(q,\theta)}_{\alpha,p,n}\mathbbm{1}_{[0,\infty)})^\wedge(\lambda-\alpha)$ for $\lambda\in (\Phi(p+q)\lor\theta,\infty)$, which together with $(\mathcal{Z}^{(q,\theta)}_{\alpha,p,0}\mathbbm{1}_{[0,\infty)})^\wedge(\lambda)=\frac{\psi(\lambda)-\psi(\theta)}{(\lambda-\theta)(\psi(\lambda)-p)}$ produces $(\mathcal{Z}^{(q,\theta)}_{\alpha,p,\infty}\mathbbm{1}_{[0,\infty)})^\wedge(\lambda)=\sum_{k=0}^\infty q^k\frac{\psi(\lambda-k\alpha)-\psi(\theta)}{(\lambda-k\alpha-\theta)(\psi(\lambda-k\alpha)-p)}\left(\prod_{l=0}^{k-1}(\psi(\lambda-l\alpha)-p)\right)^{-1}$. The starting estimate and finiteness property follow from Remark~\ref{remark:equivalent-Z}. 

\eqref{analytic:Z:ct-diif} and \eqref{analytic:Z:entire}. From its definition, $Z^{(p,\theta)}$ is continuous and $Z^{(p,\theta)}\vert_{[0,\infty)}$ is continuously differentiable. The claims of these two items, together with the statement immediately following \eqref{analytic:Z:ct-diif}, then follow via the series in $q$ of \eqref{analytic:Z:recursion}, using \eqref{eq:Esscher-W}-\eqref{eq:estimate}.
\end{proof}
Finally, we would be remiss not to point out that there is another set of scale functions pertaining to the snLp $X$, namely the exponential family $(e^{\Phi(q)\cdot})_{q\in [0,\infty)}$ associated to first passage upwards. Their analogues for pssMp are provided by Patie's \cite{pierre} scale functions, when $\alpha>0$:
\begin{definition}
Let $\alpha>0$. For $q\in [0,\infty)$, set $\II_{\alpha,p}^{(q)}:=\II\circ \log$ for the unique \cite[Theorem~3.1]{vidmar} Borel measurable locally bounded $\II:\mathbb{R}\to \mathbb{R}$ with a left-tail that is $\Phi(p)$-subexponential \cite[Definition~2.1]{vidmar} and that satisfies the convolution equation
$$\II=e^{\Phi(p)\cdot}+q\Wp\star (e^{\alpha \cdot}\II),$$
i.e.  $\II_{\alpha,p}^{(q)}(y)=y^{\Phi(p)}\sum_{k=0}^\infty\frac{(qy^{\alpha})^k}{\prod_{l=1}^k(\psi(\Phi(p)+l\alpha)-p)}$, $y\in (0,\infty)$.  When $\alpha=0$ we set $\II_{\alpha,p}^{(q)}(y)=y^{\Phi(q+p)}$, $y\in (0,\infty)$, $q\in [0,\infty)$.
\end{definition}
As already indicated, the role of these scale functions is in the solution to the first passage upwards problem, see Remark~\ref{remark:upwards}. 
\begin{question}\label{question:downward}
Perhaps curiously there seems to be no natural extension of the functions $(\II_{\alpha,p}^{(q)})_{q\in [0,\infty)}$ to the case $\alpha<0$ (so that they would continue to play their role in the solution of the first passage upwards problem). What could be said about this case 
(and hence, viz. Remark~\ref{rmk:dual}, about the first passage downward before absorption at zero 
for spectrally positive pssMp)? 
\end{question}

\section{Two-sided exit  for $Y$}\label{section:tse}
The next result corresponds to Item \eqref{intro:a} from the Introduction. 
\begin{theorem}\label{theorem:tse}
Let $\{c,d\}\subset (0,\infty)$, $c<d$, $y\in [c,d]$, $\{q,\theta\}\subset [0,\infty)$. Then:
\begin{enumerate}[(i)]
\item\label{tse:smooth} $\QQ_y\left[e^{-qT_d^+};T_d^+<T_c^-\right]=\frac{\mathcal{W}^{(qc^{\alpha})}_{\alpha,p}(y/c)}{\mathcal{W}^{(qc^\alpha)}_{\alpha,p}(d/c)}$.
\item\label{tse:non-smooth} $\QQ_y\left[e^{-qT_c^-}\left(\frac{Y_{T_c^-}}{c}\right)^\theta;T_c^-<T_d^+\right]=\mathcal{Z}^{(qc^{\alpha},\theta)}_{\alpha,p}(y/c)-\frac{\mathcal{W}^{(qc^{\alpha})}_{\alpha,p}(y/c)}{\mathcal{W}^{(qc^{\alpha})}_{\alpha,p}(d/c)}\mathcal{Z}^{(qc^{\alpha},\theta)}_{\alpha,p}(d/c)$.
\end{enumerate}
\end{theorem}
\begin{remark}
When $\alpha=p=0$ these are (modulo the $\exp$-$\log$ spatial transformation)  classical results for snLp, e.g.  \cite[Eq.~(8.11)]{kyprianou} and  \cite[3rd display on p.~4]{aiz}. 
\end{remark}
\begin{remark}\label{remark:upwards}
The first passage upward problem can in principle be seen as a limiting case (as $c\downarrow 0$) of the two-sided exit problem, however the resulting limits do not appear easy to evaluate directly. Nevertheless, the following result is known \cite[Theorem~2.1]{pierre} \cite[Theorem~13.10(ii)]{kyprianou} \cite[Example~3.2]{vidmar} for the case $\alpha>0$ and  \cite[Theorem~3.12]{kyprianou} $\alpha=0$:
$$\QQ_y[e^{-qT_d^+};T_d^+<\zeta]=\frac{\II_{\alpha,p}^{(q)}(y)}{\II_{\alpha,p}^{(q)}(d)},\quad d\in [y,\infty),y\in (0,\infty),q\in [0,\infty).$$
\end{remark}
\begin{proof}
Let $a:=\log c$, $b:=\log d$ and $x:=\log y$. 

\ref{tse:smooth}. From the Lamperti transform, the spatial homogeneity of $X$, the independence of $X$ from $\expp$, Remark~\ref{remark:equivalent-W} and \cite[Theorem~2.1]{zbigniew-bo}, we have $\QQ_y\left[e^{-qT_d^+};T_d^+<T_c^-\right]=\PP_ x[e^{-q\int_0^{\tau_b^+}e^{\alpha X_s}\dd s};\tau_b^+<\tau_a^-\land \expp]=\PP_ x[e^{-\int_0^{\tau_b^+}(qe^{\alpha X_s}+p)\dd s};\tau_b^+<\tau_a^-]=\PP_ {x-a}[e^{-\int_0^{\tau_{b-a}^+}(qe^{\alpha a}e^{\alpha X_s}+p)\dd s};\tau_{b-a}^+<\tau_0^-]=\frac{\mathcal{W}^{(qe^{\alpha a})}_{\alpha,p}(e^{x-a})}{\mathcal{W}^{(qe^{\alpha a})}_{\alpha,p}(e^{b-a})}$. 

\ref{tse:non-smooth}. Again the Lamperti transform, the spatial homogeneity of $X$ and the independence of $X$ from $\expp$ yield $\QQ_y\left[e^{-qT_c^-}\left(\frac{Y_{T_c^-}}{c}\right)^\theta;T_c^-<T_d^+\right]=\PP_ x[e^{-q\int_0^{\tau_a^-}e^{\alpha X_s}\dd s+\theta(X_{\tau_a^-}-a)};\tau_a^-<\tau_b^+\land\expp]=\PP_ {x-a}\left[e^{-\int_0^{\tau_0^-}(qe^{\alpha a}e^{\alpha X_s}+p)\dd s+\theta X_{\tau_0^-}};\tau_0^-<\tau_{b-a}^+\right]$. 
Then,  via the Esscher transform, $\QQ_y\left[e^{-qT_c^-}\left(\frac{Y_{T_c^-}}{c}\right)^\theta;T_c^-<T_d^+\right]=e^{\theta(x-a)}\PP_{x-a}^{\theta}[e^{-\int_0^{\tau_0^-}(qe^{\alpha a}e^{\alpha X_s}+p-\psi(\theta))\dd s};\tau_0^-<\tau_{b-a}^+]$. At least for $\theta\in [0,\Phi(p+q(e^{\alpha a}\land e^{\alpha b}))]$, it now follows from \cite[Theorem~2.1]{zbigniew-bo} that $$\QQ_y\left[e^{-qT_c^-}\left(\frac{Y_{T_c^-}}{c}\right)^\theta;T_c^-<T_d^+\right]=e^{\theta(x-a)}\left(\mathcal{Z}(x-a)-\frac{\mathcal{W}(x-a)}{\mathcal{W}(b-a)}\mathcal{Z}(b-a)\right),$$ where $\mathcal{W}$ and $\mathcal{Z}$ are, respectively, the unique locally bounded and Borel measurable solutions to the convolution equations 
$$\mathcal{W}=W_\theta+W_\theta\star ((qe^{\alpha a} e^{\alpha(\cdot \land (b-a))}+p-\psi(\theta))\mathbbm{1}_{[0,\infty)}\mathcal{W})$$
and
$$\mathcal{Z}=1+W_\theta\star ((qe^{\alpha a} e^{\alpha(\cdot \land (b-a)) }+p-\psi(\theta))\mathbbm{1}_{[0,\infty)}\mathcal{Z}),$$
i.e. \cite[Lemma~2.1 \& Eq.~(2.10)]{zbigniew-bo} 
$$\mathcal{W}=W_\theta^{(qe^{\alpha a}+p-\psi(\theta))}+W_\theta^{(qe^{\alpha a}+p-\psi(\theta))}\star ((qe^{\alpha a} e^{\alpha(\cdot \land (b-a))}-1)\mathbbm{1}_{[0,\infty)}\mathcal{W})$$
and 
$$\mathcal{Z}=Z^{(qe^{\alpha a}+p-\psi(\theta))}+W_\theta^{(qe^{\alpha a}+p-\psi(\theta))}\star ((qe^{\alpha a} e^{\alpha(\cdot \land (b-a)) }-1))\mathbbm{1}_{[0,\infty)}\mathcal{Z}).$$
Because these solutions are ``locally determined'', and we only need them on the interval $[0,b-a]$ (in fact only at the points $x-a$ and $b-a$), we may now drop ``$\land (b-a)$''. Moreover, multiplying both sides of the preceding two displays by $e^{\theta\cdot}$ and exploiting the relations $W^{(qe^{\alpha a}+p)}=e^{\theta\cdot}W_\theta^{(qe^{\alpha a}+p-\psi(\theta))}$ and $Z^{(qe^{\alpha a}+p,\theta)}=e^{\theta\cdot}Z_\theta^{(qe^{\alpha a}+p-\psi(\theta))}$, we find that $\QQ_y\left[e^{-qT_c^-}\left(\frac{Y_{T_c^-}}{c}\right)^\theta;T_c^-<T_d^+\right]=\left(\mathcal{Z}(x-a)-\frac{\mathcal{W}(x-a)}{\mathcal{W}(b-a)}\mathcal{Z}(b-a)\right)$, where 
$\mathcal{W}$ and $\mathcal{Z}$ are, respectively, the unique locally bounded and Borel measurable solutions to the convolution equations 
$$\mathcal{W}=W^{(qe^{\alpha a}+p)}+W^{(qe^{\alpha a}+p)}\star ((qe^{\alpha a} e^{\alpha\cdot }-1)\mathbbm{1}_{[0,\infty)}\mathcal{W})$$
and
$$\mathcal{Z}=Z^{(qe^{\alpha a}+p,\theta)}+W^{(qe^{\alpha a}+p)}\star  ((qe^{\alpha a} e^{\alpha\cdot }-1)\mathbbm{1}_{[0,\infty)}\mathcal{Z}).$$
The claim now follows from Remarks~\ref{remark:equivalent-W} and ~\ref{remark:equivalent-Z}, assuming still that $\theta\in [0,\Phi(p+q(e^{\alpha a}\land e^{\alpha b}))]$. The general case is got by analytic continuation in $q$ at fixed $\theta$.
\end{proof}
As announced in the previous section the solution to the two-sided exit problem may be used to give the excursion-theoretic relation of Lemma~\ref{lemma:fundamental}.
\begin{proof}[Proof of Lemma~\ref{lemma:fundamental}]\label{proof-conclusion}
From Theorem~\ref{theorem:tse}\ref{tse:smooth}, the notation of (the proof of) which we retain, we see that $\frac{\mathcal{W}^{(qe^{\alpha a})}_{\alpha,p}(e^{x-a})}{\mathcal{W}^{(qe^{\alpha a})}_{\alpha,p}(e^{b-a})}=\PP_ x[e^{-\int_0^{\tau_b^+}(qe^{\alpha X_s}+p)\dd s};\tau_b^+<\tau_a^-]=\PP[e^{-\int_0^{\tau_{b-x}^+}(qe^{\alpha x}e^{\alpha X_s}+p)\dd s};\tau_{b-x}^+<\tau_{a-x}^-]=\PP[e^{-\int_0^{\tau_{b-x}^+}(qe^{\alpha x}e^{\alpha X_s}+p)\mathbbm{1}_{\{\overline{X}_s=X_s\}}\dd s}\prod_{g\in \DD\cap [0,b-x]}e^{-\int_0^{\chi(\epsilon_g)}(qe^{\alpha (x+g)}e^{\alpha \epsilon_g(s)}+p)\dd s}\mathbbm{1}_{\{g+\underline{\epsilon_g}_{\infty}\geq a-x\}}]$, where the first factor in the $\PP$-expectation is only not equal to $1$ when $X$ has paths of finite variation \cite[Theorem~6.7]{kyprianou}, in which case $X$ is the difference of a strictly positive drift $\delta$ and a (non-vanishing) subordinator, and this factor is then $\PP$-a.s. equal to $e^{-\int_0^{(b-x)/\delta}(qe^{\alpha x}e^{\alpha \delta s}+p)\dd s}=e^{-(p(b-x)+q(e^{\alpha b}-e^{\alpha x})/\alpha)/\delta}$, where the expression is understood in the limiting sense when $\alpha=0$ (see e.g. \cite[proof of Theorem~4.1]{kyprianou}). As for the second factor, by the exponential formula for Ppp \cite[Theorem~4.5]{ito} its $\PP$-expectation is equal to $\PP[\prod_{g\in \DD\cap [0,b-x]}e^{-\int_0^{\chi(\epsilon_g)}(qe^{\alpha (x+g)}e^{\alpha \epsilon_g(s)}+p)\dd s}\mathbbm{1}_{\{g+\underline{\epsilon_g}_{\infty}\geq a-x\}}]=\exp\{-\int_x^{b}\nn[1-e^{-\int_0^{\chi}(qe^{\alpha g}e^{\alpha\xi_s}+p)\dd s}\mathbbm{1}_{\{g+\underline{\xi}_{\infty}\geq a\}}]\dd g\}$. 
\end{proof}
We conclude this section by using Theorem~\ref{theorem:tse} to show that the scale functions are naturally related to a family of martingales involving the process $Y$.
\begin{corollary}\label{corollary}
Let $\{q,\theta\}\subset [0,\infty)$ and $\{y,d\}\subset (0,\infty)$. Set $\GG_s:=\FF_{\phi_s}$ for $s\in [0,\infty)$. The processes $(e^{-qs}\WW(Y_s)\mathbbm{1}_{\{s<\zeta\}})_{s\in [0,\infty)}$ and  $(e^{-qs}\mathcal{Z}^{(q,\theta)}_{\alpha,p}(Y_s)\mathbbm{1}_{\{s<\zeta\}})_{s\in [0,\infty)}$ stopped at $T_1^-\land T_d^+$ are bounded c\`adl\`ag martingales under $\QQ_y$ in the filtration $\GG=(\GG_s)_{s\in [0,\infty)}$; their terminal values are, respectively, $\WW(d)e^{-qT_d^+}\mathbbm{1}_{\{T_d^+<T_1^-\}}$ and $e^{-q (T_1^-\land T_d^+)}(\mathcal{Z}^{(q,\theta)}_{\alpha,p}(d)\mathbbm{1}_{\{T_d^+<T_1^-\}}+(Y_{T_1^-})^\theta\mathbbm{1}_{\{T_1^-<T_d^+\}})$ a.s.-$\QQ_y$. 
\end{corollary}
 \begin{remark}
Recall that $Y$ enjoys the self-similarity property: the law of $(cY_{sc^{-\alpha}})_{s\in [0,\infty)}$ under $\QQ_y$ is that of $Y$ under $\QQ_{cy}$, which means that the ``$1$'' in the preceding may be generalized to a $c\in (0,\infty)$, subject to the obvious changes.
\end{remark}
\begin{remark}
When $\alpha=0$ this corollary becomes a well-known (at least for $\theta=0$) property of the scale functions of snLp, e.g. \cite[Remark~5]{akp}.
\end{remark}
\begin{remark}
One has the parallel martingales from the first passage upwards problem: when $\alpha\geq 0$, then the process $(e^{-qs}\II^{(q)}_{\alpha,p}(Y_s)\mathbbm{1}_{\{s<\zeta\}})_{s\in [0,\infty)}$ stopped at $T_d^+$ is a bounded c\`adl\`ag martingale in $\GG$ under $\QQ_y$ with terminal value $\II^{(q)}_{\alpha,p}(d)e^{-q T_d^+}\mathbbm{1}_{\{T_d^+<\zeta\}}$ a.s.-$\QQ_y$ \cite[Remark~3.1]{vidmar:pssMp}.
\end{remark}
\begin{proof}
Remark that since $\FF$ is right-continuous, then so is $\GG$ \cite[Lemma~6.3]{kallenberg}, hence $T_1^+$ and $T_d^+$ are $\GG$-stopping times. We may assume $1\leq y\leq d$. 

The assumptions on  $X$, $\expp$ and $\FF$ entail that for any $\FF$-stopping time $S$, on $\{S<\expp\}$, $\FF_S$ is independent of $((X_{S+u}-X_S)_{u\in [0,\infty)},\expp-S)$, which has the distribution of $(X,\expp)$ under $\PP$. In consequence $Y$ is Markov with life-time $\zeta$, cemetery state $\partial$, in the  filtration $\GG$, under the probabilities $(\QQ_y)_{y\in (0,\infty)}$ (note that $\zeta$ is a $\GG$-stopping time, and would be so, even if we had not assumed $\FF$ to be right-continuous).

Let now $s\in [0,\infty)$. We compute: 
$$\QQ_y[\WW(d)e^{-q T_d^+}\mathbbm{1}_{\{T_d^+<T_1^-\}}\vert\GG_s]$$\footnotesize
$$=e^{-q(T_1^-\land T_d^+)}\WW(Y_{T_1^-\land T_d^+})\mathbbm{1}_{\{T_1^-\land T_d^+\leq s\}}+\WW(d)e^{-qs}\mathbbm{1}_{\{s<T_1^-\land T_d^+\}}\QQ_{Y_{s}}[e^{-qT_d^+};T_d^+<T_1^-]\mathbbm{1}_{\{s<\zeta\}}$$\normalsize
$$=e^{-q(T_1^-\land T_d^+)}\WW(Y_{T_1^-\land T_d^+})\mathbbm{1}_{\{T_1^-\land T_d^+\leq s\}}+\WW(Y_{s})e^{-qs}\mathbbm{1}_{\{s<T_1^-\land T_d^+,s<\zeta\}}$$
$$=e^{-q(s\land T_1^-\land T_d^+)}\WW(Y_{s\land T_1^-\land T_d^+})\mathbbm{1}_{\{s\land T_1^-\land T_d^+<\zeta\}},$$ a.s.-$\QQ_y$,
where the first equality uses the following facts: that $\WW$ vanishes on $(0,1]$ when $X$ has paths of infinite variation, whereas otherwise it vanishes on $(0,1)$,  but then $X$ does not creep downwards (indeed it creeps downwards iff $\sigma^2>0$ \cite[p.~232]{kyprianou}); that $X$ has no positive jumps; the Markov property. The second equality of the preceding display follows from Theorem~\ref{theorem:tse}\ref{tse:smooth}.

Similarly 
$$\QQ_y[e^{-q (T_1^-\land T_d^+)}(\mathcal{Z}^{(q,\theta)}_{\alpha,p}(d)\mathbbm{1}_{\{T_d^+<T_1^-\}}+(Y_{T_1^-})^\theta\mathbbm{1}_{\{T_1^-<T_d^+\}})\vert\GG_s]$$\footnotesize
$$=e^{-q (T_1^-\land T_d^+)}\mathcal{Z}^{(q,\theta)}_{\alpha,p}(Y_{T_1^-\land T_d^+})\mathbbm{1}_{\{T_1^-\land T_d^+\leq s\}}+e^{-qs}\QQ_{Y_{s}}[e^{-q (T_1^-\land T_d^+)}(\mathcal{Z}^{(q,\theta)}_{\alpha,p}(d)\mathbbm{1}_{\{T_d^+<T_1^-\}}+(Y_{T_1^-})^\theta\mathbbm{1}_{\{T_1^-<T_d^+\}})]\mathbbm{1}_{\{s<\zeta,s<T_1^-\land T_d^+\}}$$
$$=e^{-q (T_1^-\land T_d^+)}\mathcal{Z}^{(q,\theta)}_{\alpha,p}(Y_{T_1^-\land T_d^+})\mathbbm{1}_{\{T_1^-\land T_d^+\leq s\}}+e^{-qs}\left(\frac{\WW(Y_{s})}{\WW(d)}\mathcal{Z}^{(q,\theta)}_{\alpha,p}(d)+\mathcal{Z}^{(q,\theta)}_{\alpha,p}(Y_{s})-\frac{\mathcal{W}^{(q)}_{\alpha,p}(Y_{s})}{\mathcal{W}^{(q)}_{\alpha,p}(d)}\mathcal{Z}^{(q,\theta)}_{\alpha,p}(d)\right)\mathbbm{1}_{\{s<\zeta,s<T_1^-\land T_d^+\}}$$\normalsize
$$=e^{-q (T_1^-\land T_d^+\land s)}\mathcal{Z}^{(q,\theta)}_{\alpha,p}(Y_{T_1^-\land T_d^+\land s})\mathbbm{1}_{\{s\land T_1^-\land T_d^+<\zeta\}},$$ a.s.-$\QQ_y$, where the first equality uses  $\mathcal{Z}^{(q,\theta)}_{\alpha,p}(y)=y^\theta$ for $y\in (0,1]$, that $X$ has no positive jumps and the Markov property, while the second equality follows from Theorem~\ref{theorem:tse}.
\end{proof}
\section{Mixed first passage for $(Y,R)$}\label{section:drawdowns}
Our last result concerns Items \eqref{intro:b} and \eqref{intro:c} from the Introduction. To streamline the expressions that we will obtain, we specify one last piece of notation, to wit
\begin{equation}
\mathsf{I}^{(q)}_{\alpha,p}(y,d;r):=\int_y^{d}r(z)\frac{(\mathcal{W}^{(q(z/r(z))^{\alpha})}_{\alpha,p})'_+ (r(z))}{\mathcal{W}^{(q(z/r(z))^{\alpha})}_{\alpha,p}(r(z))}\frac{\dd z}{z},\quad \{y,d\}\subset (0,\infty),\, y\leq d, 
\end{equation}
where further $q\in [0,\infty)$ and $r\in  \mathcal{B}_{(0,\infty)}/\mathcal{B}_{(1,\infty)}$.
\begin{theorem}\label{theorem:drawdown}
Let $\{y,d\}\subset (0,\infty)$, $y\leq d$, $r\in \mathcal{B}_{(0,\infty)}/\mathcal{B}_{(1,\infty)}$, $\{q,\theta,\gamma\}\subset [0,\infty)$. Assume $\Sigma_r$ is measurable.  Then:
\begin{enumerate}[(i)]
\item\label{drawdown:i} $\QQ_y\left[e^{-qT_d^+};T_d^+<\Sigma_r\right]=e^{-\mathsf{I}^{(q)}_{\alpha,p}(y,d;r)}$.
\item\label{drawdown:ii} $\QQ_y\left[e^{-q \Sigma_r-\gamma L_{\Sigma_r}}\left(\frac{r(\overline{Y}_{\Sigma_r})}{R_{\Sigma_r}}\right)^\theta;\Sigma_r<T_d^+ \right]$ \footnotesize
$$=\bigintsss_y^{d}r(z)e^{-\mathsf{I}^{(q+\gamma)}_{\alpha,p}(y,z;r)}\left[\frac{(\mathcal{W}^{(q(z/r(z))^{\alpha})}_{\alpha,p})^\prime_+(r(z))}{\mathcal{W}^{(q(z/r(z))^{\alpha})}_{\alpha,p}(r(z))}\mathcal{Z}^{(q(z/r(z))^{\alpha},\theta)}_{\alpha,p}(r(z))-\mathcal{Z}^{(q(z/r(z))^{\alpha},\theta)\prime}_{\alpha,p}(r(z))\right]\frac{\dd z}{z}.$$ \normalsize
\end{enumerate}
\end{theorem}
\begin{remark}
By the d\'ebut theorem certainly $\Sigma_r$ is universally measurable. 
\end{remark}
\begin{remark}
For the special case $\alpha=p=0$  see (modulo the trivial $\exp-\log$ spatial transformation): \cite[Theorem~1]{mp} (and also, for \ref{drawdown:i}, in case $X$ has paths of infinite variation, \cite[Eq.~(3.1)]{llz}) when $r$ is constant; \cite[Proposition~3.1]{li-vu} when $r$ is not necessarily constant.
\end{remark}
\begin{remark}\label{rmk:density}
Note that $\{\Sigma_r<T_d^+\}=\{\overline{Y}_{\Sigma_r}\leq d\}$ on $\{\Sigma_r<\zeta\}$. Consequently, by a monotone class argument, \ref{drawdown:ii}  implies 
that $\QQ_y\left[e^{-q \Sigma_r-\gamma L_{\Sigma_r}}\left(\frac{r(\overline{Y}_{\Sigma_r})}{R_{\Sigma_r}}\right)^\theta f(\overline{Y}_{\Sigma_r});\Sigma_r<\zeta\right]=$\footnotesize
$$\bigintsss_y^\infty \!\!\!\!\!\!r(z)e^{-\mathsf{I}^{(q+\gamma)}_{\alpha,p}(y,z;r)}\left[\frac{(\mathcal{W}^{(q(z/r(z))^{\alpha})}_{\alpha,p})^\prime_+(r(z))}{\mathcal{W}^{(q(z/r(z))^{\alpha})}_{\alpha,p}(r(z))}\mathcal{Z}^{(q(z/r(z))^{\alpha},\theta)}_{\alpha,p}(r(z))-\mathcal{Z}^{(q(z/r(z))^{\alpha},\theta)\prime}_{\alpha,p}(r(z))\right]f(z)\frac{\dd z}{z}$$ \normalsize for $f\in \mathcal{B}_{(0,\infty)}/\mathcal{B}_{[-\infty,\infty]}$, in the sense that the left-hand side is well-defined iff the right-hand side is so, in which case they are equal.
 \end{remark}
\begin{remark}
Let $p=0$ and set $b:=\log(d)$, $x:=\log(y)$ and $s:=\log\circ r\circ \exp$. Since $p=0$ implies $\Sigma_r\land T_d<\infty$ a.s., then from \ref{drawdown:i}, or else (also as a check) using the fundamental theorem of calculus from \ref{drawdown:ii}, we see that $\QQ_y(\Sigma_r<T_d^+)=\PP_x(\sigma_s<\tau_b^+)=1-\exp\left(-\int_x^{b}\frac{W'_+ (s(w))}{W(s(w))}\dd w\right)$. In particular by monotone convergence $\QQ_y(\Sigma_r<\zeta)=\PP_x(\sigma_s<\infty)=1$ iff $\int_x^\infty\frac{W'_+ (s(w))}{W(s(w))}\dd w=\infty$.
\end{remark}
\begin{proof}
Let $b:=\log(d)$, $x:=\log(y)$ and $s:=\log\circ r\circ \exp$. 

\ref{drawdown:i}. Of course $\QQ_y[e^{-qT_d};T_d^+<\Sigma_r]=\PP_x[e^{-q\int_0^{\tau_b^+}e^{\alpha X_t}\dd t};\tau_b^+<\sigma_s\land \expp]=\PP[e^{-\int_0^{\tau_{b-x}^+}(qe^{\alpha x}e^{\alpha X_t}+p)\dd t};\tau_{b-x}^+<\sigma_{s(x+ \cdot)}]$. By the exponential formula for Ppp  it is equal to
$$e^{-(p(b-x)+qe^{\alpha x}(e^{\alpha (b-x)}-1)/\alpha)/\delta}\exp\left\{-\int_0^{b-x}\nn[1-e^{-\int_0^{\chi}(qe^{\alpha x}e^{\alpha g}e^{\alpha\xi_v}+p)\dd v}\mathbbm{1}_{\{\underline{\xi}_{\infty}\geq -s(x+g)\}}]\dd g\right\}$$
(recall that in the infinite variation case we take $\delta=\infty$). Using Proposition~\ref{proposition:analytic}\eqref{analytic:c} it is now straightforward to check that this agrees precisely with \ref{drawdown:i}.

\ref{drawdown:ii}. We begin by noting that, using the Esscher transform as in the proof of Theorem~\ref{theorem:tse}\ref{tse:non-smooth},\small $$\QQ_y\left[e^{-q \Sigma_r-\gamma L_{\Sigma_r}}\left(\frac{r(\overline{Y}_{\Sigma_r})}{R_{\Sigma_r}}\right)^\theta;\Sigma_r<T_d^+ \right]=\PP_x[e^{-\int_0^{\sigma_s}(qe^{\alpha X_t}+p)\dd t-\gamma \int_0^{G_{\sigma_s}}e^{\alpha X_t}\dd t-\theta(D_{\sigma_s}-s(\overline{X}_{\sigma_s}))};\sigma_s<\tau_b^+]$$\normalsize
$$=e^{\theta x}\PP_x^\theta[e^{-\int_0^{\sigma_s}(qe^{\alpha X_t}+p-\psi(\theta))\dd t-\gamma \int_0^{G_{\sigma_s}}e^{\alpha X_t}\dd t-\theta(\overline{X}_{\sigma_s}-s(\overline{X}_{\sigma_s}))};\sigma_s<\tau_b^+]$$
$$=\PP^\theta[e^{-\int_0^{\sigma_s}(qe^{\alpha x}e^{\alpha X_t}+p-\psi(\theta))\dd t-\gamma \int_0^{G_{\sigma_s}}e^{\alpha x}e^{\alpha X_t}\dd t-\theta(\overline{X}_{\sigma_s}-s(x+\overline{X}_{\sigma_s}))};\sigma_{s(x+\cdot)}<\tau_{b-x}^+]$$
$$=\PP^\theta\Big[\sum_{g\in \DD\cap [0,b-x]}\mathbbm{1}_{[0,\sigma_{s(x+\cdot)}]}(\tau_{g-}^+)\exp\left(-\int_0^{\tau_{g-}^+}((q+\gamma)e^{\alpha x}e^{\alpha X_t}+p-\psi(\theta))\dd t-\theta(g-s(x+g))\right)$$
$$\exp\left(-\int_0^{S^-_{-s(x+g)}(\epsilon_g)}(qe^{\alpha (x+g)}e^{\alpha \epsilon_g(t)}+p-\psi(\theta))\dd t\right)\mathbbm{1}_{\{S_{-s(x+g)}^-(\epsilon_g)<\chi(\epsilon_g)\}}\Big].$$  By the compensation formula \cite[p.~7]{bertoin} for the Ppp of excursions of $X$ from the maximum this becomes 
$$\PP^\theta\Bigg[\int_0^{\overline{X}_\infty\land (b-x)}\mathbbm{1}_{[0,\sigma_{s(x+\cdot)}]}(\tau_{g-}^+)e^{-\int_0^{\tau_{g-}^+}((q+\gamma)e^{\alpha x}e^{\alpha X_t}+p-\psi(\theta))\dd t-\theta(g-s(x+g))}$$
$$\quad\quad\quad\quad\quad\nn^\theta\left[\exp\left(-\int_0^{S^-_{-s(x+g)}}(qe^{\alpha (x+g)}e^{\alpha \xi_t}+p-\psi(\theta))\dd t\right);S_{-s(x+g)}^-<\chi\right]\dd g\Bigg]$$
$$=\int_0^{b-x}e^{-\theta (m-s(x+m))}\PP^\theta[e^{-\int_0^{\tau_m^+}((q+\gamma)e^{\alpha x}e^{\alpha X_t}+p-\psi(\theta))\dd t};\tau_m^+<\sigma_{s(x+\cdot)}]$$ $$\quad\quad\quad\quad\quad\quad\quad\nn^\theta\left[\exp\left(-\int_0^{S^-_{-s(x+m)}}(qe^{\alpha (x+m)}e^{\alpha \xi_t}+p-\psi(\theta))\dd t\right);S_{-s(x+m)}^-<\chi\right]\dd m,$$
where the superscript $\theta$ in $\nn^\theta$ indicates that the quantity pertains to the Esscher transformed process.

Now we know already from the previous part that, for $m\in [0,\infty)$,
$$\PP^\theta[e^{-\int_0^{\tau_m^+}((q+\gamma)e^{\alpha x}e^{\alpha X_t}+p-\psi(\theta))\dd t};\tau_m^+<\sigma_{s(x+\cdot)}]=e^{\theta m }\PP[e^{-\int_0^{\tau_m^+}((q+\gamma)e^{\alpha x}e^{\alpha X_t}+p)\dd t};\tau_m^+<\sigma_{s(x+\cdot)}]$$
$$=e^{\theta m-\mathsf{I}^{(q+\gamma)}_{\alpha,p}(y,ye^m;r)}.$$

Finally the results of  \cite[Proposition~2.2]{doney-conditioning} and Theorem~\ref{theorem:tse}\ref{tse:non-smooth} imply that, for some multiplicative constant $\mathsf{k}\in (0,\infty)$, which depends only on the characteristics of $X$, and then all $m\in [0,\infty)$,\footnotesize
$$\nn^\theta\left[\exp\left(-\int_0^{S^-_{-s(x+m)}}(qe^{\alpha (x+m)}e^{\alpha \xi_t}+p-\psi(\theta))\dd t\right);S_{-s(x+m)}^-<\chi\right]$$
$$=\mathsf{k}\lim_{z\downarrow  0}\frac{\PP_{-z}^\theta\left[\exp\left(-\int_0^{\tau_{-s(x+m)}^-}(qe^{\alpha (x+m)}e^{\alpha X_t}+p-\psi(\theta))\dd t\right);\tau_{-s(x+m)}^-<\tau_0^+\right]}{z}$$\normalsize
$$=\mathsf{k}e^{-\theta s(x+m)}\lim_{z\downarrow 0}\frac{\PP_{s(x+m)-z}\left[\exp\left(-\int_0^{\tau_{0}^-}(qe^{\alpha (x+m-s(x+m))}e^{\alpha X_t}+p)\dd t+\theta X_{\tau_0^-}\right);\tau_{0}^-<\tau_{s(x+m)}^+\right]}{z}$$\footnotesize
$$=\mathsf{k}e^{-\theta s(x+m)}\lim_{z\downarrow 0}\frac{\mathcal{Z}^{(q(ye^m/r(ye^m))^{\alpha},\theta)}_{\alpha,p}(r(ye^m)e^{-z})-\frac{\mathcal{W}^{(q(ye^m/r(ye^m))^{\alpha})}_{\alpha,p}(r(ye^m)e^{-z})}{\mathcal{W}^{(q(ye^m/r)^{\alpha})}_{\alpha,p}(r(ye^m))}\mathcal{Z}^{(q(ye^m/r(ye^m))^{\alpha},\theta)}_{\alpha,p}(r(ye^m))}{z}$$
$$=\mathsf{k}r(ye^m)e^{-\theta s(x+m)}\left[\frac{(\mathcal{W}^{(q(ye^m/r(ye^m))^{\alpha})}_{\alpha,p})^\prime_+(r(ye^m))}{\mathcal{W}^{(q(ye^m/r(ye^m))^{\alpha})}_{\alpha,p}(r)}\mathcal{Z}^{(q(ye^m/r(ye^m))^{\alpha},\theta)}_{\alpha,p}(r(ye^m))-\mathcal{Z}^{(q(ye^m/r(ye^m))^{\alpha},\theta)\prime}_{\alpha,p}(r(ye^m))\right].$$\normalsize
Bringing everything together, the result follows: one establishes that $\mathsf{k}=1$ for instance by plugging in $q=\theta=\gamma=p=0$, $r$ constant, and comparing to part \ref{drawdown:i} (it is well-known that for a constant $s$, $\sigma_s<\infty$ a.s., so that, when $p=0$, also $\Sigma_r<\infty$ a.s.).
\end{proof}
\begin{corollary}\label{corollary:drawdowns}
Let $y\in (0,\infty)$, $r\in \mathcal{B}_{(0,\infty)}/\mathcal{B}_{(1,\infty)}$, $\{q,\gamma,\theta\}\subset [0,\infty)$. Then:
$\QQ_y\left[e^{-q \Sigma_r-\gamma L_{\Sigma_r}}\left(\frac{r(\overline{Y}_{\Sigma_r})}{R_{\Sigma_r}}\right)^\theta;\Sigma_r<\zeta \right]$ 
$$=\bigintsss_y^{\infty}r(l)e^{-\mathsf{I}^{(q+\gamma)}_{\alpha,p}(y,l;r)}\left[\frac{(\mathcal{W}^{(q(l/r(l))^{\alpha})}_{\alpha,p})^\prime_+(r(l))}{\mathcal{W}^{(q(l/r(l))^{\alpha})}_{\alpha,p}(r(l))}\mathcal{Z}^{(q(l/r(l))^{\alpha},\theta)}_{\alpha,p}(r(l))-(\mathcal{Z}^{(q(l/r(l))^{\alpha},\theta)}_{\alpha,p})^\prime(r(l))\right]\frac{\dd l}{l};$$ \normalsize
in particular $\QQ_y\left[e^{-q \Sigma_r};\Sigma_r<\zeta \right]$
$$
=\bigintsss_y^{\infty}r(l)e^{-\mathsf{I}^{(q)}_{\alpha,p}(y,l;r)}\left[\frac{(\mathcal{W}^{(q(l/r(l))^{\alpha})}_{\alpha,p})^\prime_+(r(l))}{\mathcal{W}^{(q(l/r(l))^{\alpha})}_{\alpha,p}(r(l))}\mathcal{Z}^{(q(l/r(l))^{\alpha})}_{\alpha,p}(r(l))-\mathcal{Z}^{(q(l/r(l))^{\alpha})\prime}_{\alpha,p}(r(l))\right]\frac{\dd l}{l}.
$$
\end{corollary}
\begin{proof}
Apply monotone convergence in Theorem~\ref{theorem:drawdown}\ref{drawdown:ii}.
\end{proof}

\section{An application to a trailing stop-loss selling strategy}\label{section:conclusion} 
We give now some applied flavor to the above results. 

Suppose indeed that, for $y\in (0,\infty)$, we interpret $Y$ under $\QQ_y$ as the price of a risky asset with initial price $y$.  It is natural to, and we will exclude the possibility when $Y$ reaches $\infty$ in finite time with a positive probability (it can only happen when $\alpha<0$, and $p>0$ or else $X$ drifts to $\infty$).

Suppose furthermore that, having bought the asset at time zero, we pursue, for an $r\in (1,\infty)$, a trailing stop loss exit policy to sell the asset once its price has dropped from the running maximum for the first time by strictly more than $(100 (1-r^{-1}))\%$, viz. at the time $\Sigma_r$. This is a  reasonable, psychologically appealing, trading strategy that limits the maximum loss relative to the trailing maximum. See the recent paper \cite{zhang-leung} where such policies and their offspring were studied in the context of diffusions (this paper also gives an overview of the trailing stop literature in quantitative finance).

Let now $q\in [0,\infty)$ be an impatience/discounting parameter. Then the expected discounted payout on liquidation of the asset is given by Remark~\ref{rmk:density}: \footnotesize
$$ \QQ_y[e^{-q\Sigma_r}Y_{\Sigma_r};\Sigma_r<\zeta]=\bigintsss_y^\infty e^{-r\bigintssss_y^{z}\frac{(\mathcal{W}^{(q(v/r)^{\alpha})}_{\alpha,p})^\prime_+ (r)}{\mathcal{ W}^{(q(v/r)^{\alpha})}_{\alpha,p}(r)}\frac{\dd v}{v}}\left[\frac{(\mathcal{W}^{(q(z/r)^{\alpha})}_{\alpha,p})^\prime_+(r)}{\mathcal{W}^{(q(z/r)^{\alpha})}_{\alpha,p}(r)}\mathcal{Z}^{(q(z/r)^{\alpha},1)}_{\alpha,p}(r)-\mathcal{Z}^{(q(z/r)^{\alpha},1)\prime}_{\alpha,p}(r)\right]\dd z.$$\normalsize
It would for instance be interesting to see whether an optimal $r$ exists that maximizes this expectation, and if so, what that $r$ is. 

This and (more importantly) Questions~\ref{question:drawup},~\ref{question:csbp} and~\ref{question:downward} are left open to future research -- with the hope that the above general theory will go some way towards fostering further developments in this area.
\bibliographystyle{plain}
\bibliography{pssMp}

\end{document}